\documentclass[reqno,oneside,letterpaper]{amsart}
\usepackage{mathrsfs}
\usepackage{amssymb}
\usepackage{amsbsy}
\usepackage[utf8]{inputenc}
\usepackage{amsmath, amssymb,bm, cases, mathtools, thmtools}
\usepackage{verbatim}
\usepackage{graphicx}\graphicspath{{figures/}}
\usepackage{multicol}
\usepackage{tabularx}
\usepackage[usenames,dvipsnames]{xcolor}
\usepackage{mathrsfs}
\usepackage{url}
\usepackage[normalem]{ulem}
\usepackage{xstring}
\usepackage{enumitem}

\usepackage[%
    minnames=1,maxnames=99,maxcitenames=3,
    style=alphabetic,
    doi=false,url=false,
    firstinits=true,hyperref,natbib,backend=bibtex]{biblatex}
\renewbibmacro{in:}{%
  \ifentrytype{article}{}{\printtext{\bibstring{in}\intitlepunct}}}
\bibliography{biblio}

\usepackage[colorlinks,citecolor=blue,urlcolor=blue,linkcolor=RawSienna]{hyperref}
\usepackage{hypernat}
\usepackage{datetime}

\DeclareMathAlphabet\EuRoman{U}{eur}{m}{n}
\SetMathAlphabet\EuRoman{bold}{U}{eur}{b}{n}

\usepackage{euscript,microtype}
\usepackage[capitalize]{cleveref}

\crefname{assumption}{Assumption}{Assumptions}
\crefname{claim}{Claim}{Claims}

\makeatletter
\let\reftagform@=\tagform@
\def\tagform@#1{\maketag@@@{\ignorespaces\textcolor{gray}{(#1)}\unskip\@@italiccorr}}
\renewcommand{\eqref}[1]{\textup{\reftagform@{\ref{#1}}}}
\makeatother

\definecolor{WowColor}{rgb}{.75,0,.75}
\definecolor{SubtleColor}{rgb}{0,0,.50}

\newcounter{margincounter}

\declaretheorem[style=plain,numberwithin=section,name=Theorem]{theorem}
\declaretheorem[style=plain,sibling=theorem,name=Lemma]{lemma}

\declaretheorem[style=plain,sibling=theorem,name=Claim]{claim}

\declaretheorem[style=definition,sibling=theorem,name=Definition]{definition}

\declaretheorem[style=definition,sibling=theorem,name=Example]{example}

\declaretheorem[style=definition,name=Condition]{conditionINNER}
\newenvironment{condition}[1]
 {\renewcommand{\theconditionINNER}{#1}\conditionINNER}
 {\endconditionINNER}
\crefformat{conditionINNER}{#2(#1)#3}
\crefmultiformat{conditionINNER}
  {(#2#1#3)}
  { and~(#2#1#3)}
  {, (#2#1#3)}
  { and~(#2#1#3)}
\Crefformat{conditionINNER}{Condition~#2(#1)#3}
\Crefmultiformat{conditionINNER}
  {Conditions~(#2#1#3)}
  { and~(#2#1#3)}
  {, (#2#1#3)}
  { and~(#2#1#3)}

\declaretheoremstyle[
    spaceabove=-6pt,
    spacebelow=6pt,
    headfont=\normalfont\bfseries,
    bodyfont = \normalfont,
    postheadspace=1em,
    qed=$\square$,
    headpunct={{}}]{myproofstyle}

\numberwithin{equation}{section}
\numberwithin{theorem}{section}

\usepackage{amssymb}
\usepackage{mathrsfs}
\usepackage{leftidx}

\usepackage{setspace}

\def\[#1\]{\begin{align}#1\end{align}}
\def\*[#1\]{\begin{align*}#1\end{align*}}

\newcommand{\argdot}{\cdot}

\newcommand{\Rationals}{\mathbb{Q}}
\newcommand{\Reals}{\mathbb{R}}
\newcommand{\Nats}{\mathbb{N}}

\newcommand{\NNReals}{\Reals_{\ge 0}}
\newcommand{\PosReals}{\Reals_{> 0}}

\DeclareMathOperator{\conv}{conv}

\newcommand{\dee}{\mathrm{d}}

\DeclareMathOperator*{\newlim}{\mathrm{lim}\vphantom{\mathrm{infsup}}}
\DeclareMathOperator*{\newmin}{\mathrm{min}\vphantom{\mathrm{infsup}}}
\DeclareMathOperator*{\newmax}{\mathrm{max}\vphantom{\mathrm{infsup}}}
\DeclareMathOperator*{\newinf}{\mathrm{inf}\vphantom{\mathrm{infsup}}}
\DeclareMathOperator*{\newsup}{\mathrm{sup}\vphantom{\mathrm{infsup}}}
\renewcommand{\lim}{\newlim}
\renewcommand{\min}{\newmin}
\renewcommand{\max}{\newmax}
\renewcommand{\inf}{\newinf}
\renewcommand{\sup}{\newsup}

\newcommand{\defn}[1]{\emph{#1}}

\newcommand{\cF}{\mathcal F}

\newcommand{\BorelSets}[1]{\mathcal{B}(#1)}
\newcommand{\NSE}[1]{{^{*}#1}}
\newcommand{\ST}{\mathsf{st}}

\newcommand{\AS}{A}

\newcommand{\PowerSet}{\mathscr{P}}
\newcommand{\HReals}{\NSE{\Reals}}

\newcommand{\Model}{P}

\newcommand{\Loss}{\ell}

\newcommand{\NS}[1]{\mathrm{NS}(#1)}

\newcommand{\Risk}{r}

\newcommand{\cS}{\mathcal{S}}
\newcommand{\cD}{\mathcal{D}}

\newcommand{\cA}{\mathcal{A}}

\newcommand{\cC}{\mathcal{C}}

\newcommand{\IFuncs}[2]{{#1}^{#2}}

\newcommand{\FiniteSubsets}[1]{{#1}^{[< \infty]}}

\newcommand{\vzero}{\mathbf{0}}
\newtheorem{open problem}{Open Problem}

\newcommand{\Loeb}[1]{\overline{#1}}

\newcommand{\ProbMeasures}[1]{\mathcal{M}_1(#1)}

\newcommand{\gY}{Y}

\newcommand{\interior}[1]{%
  {\kern0pt#1}^{\mathrm{o}}%
}

\newcommand{\FRRE}{\mathcal D}

\newcommand{\FiniteRiskEstimators}{\FRRE_0}
\newcommand{\refproof}[1]{See \cref{#1} for \IfSubStr{#1}{,}{proofs}{a proof}. }

\newcommand{\TTheta}{T_{\Theta}}

\newcommand{\Ext}[1]{^{\sigma}#1}

\newcommand{\sdelta}{\NSE{\!\delta}}
\newcommand{\sRisk}{\NSE{\!\Risk}}

\newcommand{\IP}[2]{\langle #1, #2\rangle}

\newif\iflongform
\longformtrue

\iflongform

\else

\fi

\newcommand{\RS}[1]{\cS^{#1}}

\makeatletter
\providecommand*{\toclevel@definition}{0}
\providecommand*{\toclevel@theorem}{0}
\providecommand*{\toclevel@lemma}{0}
\makeatother

\usepackage{leftidx}
\title
{
Statistical minimax theorems via nonstandard analysis %
}

\newcommand{\pd}[1]{{#1}_p}
\newcommand{\ipd}[1]{{#1}^{p}}
\newcommand{\FM}[1]{\mathcal{M}^{\mathsf{f}}_{1}(#1)}
\newcommand{\PM}[1]{\mathcal{M}_{1}(#1)}
\newcommand{\Lip}[2]{\mathcal{L}_{#1}(#2)}
\newcommand{\comp}[1]{\hat{#1}}
\newcommand{\topology}{\mathcal{T}}

\newcommand{\EReals}{\overline{\Reals}}
\DeclareMathOperator{\EST}{\mathsf{\bar{st}}}

\renewcommand{\subset}{\subseteq}

\begin{document}

\author[Duanmu]{Haosui Duanmu}
\address{Haosui Duanmu, Harbin Institute of Technology}

\author[Roy]{Daniel M.~Roy}
\address{Daniel M.~Roy, University of Toronto}

\author[Schrittesser]{David Schrittesser}
\address{David Schrittesser, Harbin Institute of Technology \emph{and} University of Toronto}

\begin{abstract}

For statistical decision problems with finite parameter space, it is well-known that the upper value (minimax value) agrees with the lower value (maximin value). Only under a generalized notion of  prior does such an equivalence carry over to the case infinite parameter spaces, provided nature can play a prior distribution and the statistician can play a randomized strategy. 
Various such extensions of this classical result have been established, 
but they are subject to technical conditions such as compactness of the parameter space or continuity of the risk functions.  
Using nonstandard analysis, we prove a minimax theorem for arbitrary statistical decision problems.
Informally, we show that for every statistical decision problem,
the standard upper value equals the lower value when the $\sup$ is taken over the collection of all internal priors, which may assign infinitesimal probability to (internal) events.
Applying our nonstandard minimax theorem, we derive several standard minimax theorems: 
a minimax theorem on compact parameter space with continuous risk functions, a finitely additive minimax theorem with bounded risk functions and a minimax theorem on totally bounded metric parameter spaces with Lipschitz risk functions.

\end{abstract}

\maketitle

\singlespace

\tableofcontents

\section{Introduction}
Consider a statistical decision problem with parameter space $\Theta$ and let $\PM{\Theta}$ denote the set of countably additive probability measures on $\Theta$. 
For every decision procedure $\delta$, every parameter $\theta\in \Theta$ and every $\pi\in \PM{\Theta}$, we write $\Risk(\theta,\delta)$ to denote the risk of $\delta$ with respect to $\theta$; we write $\Risk(\pi,\delta)$ to denote the Bayes risk of $\delta$ with respect to $\pi$.  
Given a set $\FRRE$ of decision procedures, it is natural to ask: Under which conditions 
does the following equality obtain?  %
\[\label{minimaxeq}
\inf_{\delta\in \FRRE}\sup_{\theta\in \Theta}\Risk(\theta,\delta)=\sup_{\pi\in \PM{\Theta}}\inf_{\delta\in \FRRE}\Risk(\pi,\delta)
\]
that is, when is the \emph{upper value} or \emph{minimax value} (the left-hand side) equal to the \emph{lower value} or \emph{maximin value} (the right-hand side)?
A decision procedure $\delta_0$ is \emph{minimax} if $\sup_{\theta\in \Theta}\Risk(\theta,\delta_0)=\inf_{\delta\in \FRRE}\sup_{\theta\in \Theta}\Risk(\theta,\delta)$. Thus, a minimax procedure is one which achieves the infinum on the left hand side of \cref{minimaxeq}; in absence of such a procedure, one may consider a sequence of decision procedures which achieve the minimax value in the limit. 
Dually, when \cref{minimaxeq} holds we call a prior $\pi$ \emph{least favorable prior} if a minimax decision procedure is optimal under $\pi$. 
More generally, we may consider a sequence of priors with the requisite limit behavior, called a least favorable sequence. 
In situations for which no prior information is available, the available minimax decision procedures are often suggested for cautious individuals, since they offer maximum protection in the worst-case scenario. 

In statistical decision theory, when one speaks of the (or a) mimimax theorem, one usually means a theorem stating that \cref{minimaxeq} holds in some particular context.
We shall refer to \cref{minimaxeq} as the \emph{minimax equality}.
It is well-known that \cref{minimaxeq} holds for statistical decision problems with finite parameter space, provided nature can play a prior distribution and the statistician can play a randomized strategy.
As the literature stands (\citep{Wald47,Wald49,LeCam55,oddsmaking}), for statistical decision problems with infinite parameter space, the validity of \cref{minimaxeq} is subject to technical conditions which limit its applicability. 
Given the plethora of technical conditions, it would not be unreasonable to presume that \cref{minimaxeq} may fail in a general setting. 

In the setting of nonstandard analysis on the other hand, we find that the standard upper value agrees with the nonstandard lower value when the supremum in the right-hand side of \cref{minimaxeq} is taken over the collection of internal priors, which may assign infinitesimal probability to (internal) events. 
This equivalence holds in complete generality;
thus, if one is willing to adopt internal priors, \cref{minimaxeq} holds for arbitrary decision problems without technical conditions such as compactness of the parameter space or continuity of the risk functions. 

By using this nonstandard minimax theorem as a master theorem, we are able to derive several standard minimax theorems. 
Thus, the present paper is another example (following \citep{DR2017, DRS2021}) of nonstandard analysis deriving foundational results in statistical decision theory. 

\subsection{Nonstandard Analysis}
Rather than working in the standard mathematical framework, we work within the normal Bayesian theory but carry out that work using nonstandard analysis, a powerful machinery first developed by \citet{AR65} using ideas from mathematical logic. These tools have recently been used to solve some fundamental problems in statistics. \citet{DR2017} obtain a general complete class theorem connecting Bayesian and frequentist optimality. \citet{nscredible} establish the existence of matching priors for statistical problems with compact parameter space, generalizing a result of \citet{muller2016coverage}. 
Duanmu, Roy, and Schrittesser \citet{DRS2021} establish an exact characterization of admissibility in arbitrary decision problems.

Generally speaking, nonstandard analysis offers powerful tools for extending existence results in finite settings to continuous ones. Moreover, almost every existence result coming from this approach automatically produces a convergence theorem. 
All this suggest that there are many more potential applications of nonstandard analysis within statistics. 

The utility of using nonstandard models stems from the use of hyperfinite objects, which are infinite objects but nonetheless possess exactly the same properties (in the requisite internal logic) as do finite objects. Under moderate conditions, hyperfinite objects are linked to standard objects via suitable maps from the nonstandard model to the standard model. Thus, in many cases, nonstandard analysis is a natural tool for extending known results in finite settings to a more general setting. 
In this paper, we use the finite version of the minimax theorem in concert with nonstandard priors with hyperfinite support to show that 
\[
\inf_{\delta\in \FRRE}\sup_{\theta\in \Theta}\Risk(\theta,\delta)=\sup_{\Pi\in \NSE{\PM{\Theta}}}\inf_{\delta\in \FRRE}\EST(\Risk(\Pi,\NSE{\delta}))
\]
for arbitrary statistical decision problems. The left side is the standard minimax value while the right side is the nonstandard maximin value over the collection of all nonstandard priors. 
Under suitable regularity conditions, we can push down nonstandard priors to obtain standard priors, maintaining the desirable properties of their nonstandard source. 
In this manner, we establish several standard minimax theorems based on the push-down of nonstandard priors. 

It is worth mentioning that \citet{teddy2017} establish a very general complete class theorem as well as a minimax theorem using finitely additive priors and finitely additive randomization. 
This suggests a deep connection between nonstandard decision theory and decision theory based on finitely additive priors. 

\subsection{Overview of the paper}
In \cref{secpre},
we introduce basic notions and key results in standard statistical decision theory:
randomized decision procedures, minimax decision procedures, least favorable priors,
and the minimax theorem for statistical decision problems with finite parameter space.
Classic treatments can be found in
\citep{Ferguson} and \citep{BG54}, the latter emphasizing the connection with game theory, but restricting itself to finite discrete spaces.
A modern treatment can be found in \citep{LC98}. 

In \cref{secnsmini}, we establish a completely general nonstandard minimax theorem. Using nonstandard analysis in concert with the classical hyperplane separation theorem, we show that for every statistical decision problem, the standard minimax value is equivalent to the nonstandard maximin value over the collection of all nonstandard priors. We also establish some connections between nonstandard minimaxity and a novel notion of Bayesian optimality. In particular, we show that, a minimax decision procedure has infinitesimal excess Bayes risk with respect to some nonstandard prior. 

In \cref{pushdownsec}, we derive three standard minimax theorems from the nonstandard minimax theorem. Namely, we show that
\begin{enumerate}
\item 
For statistical decision problems with compact parameter space and upper semi-continuous risk functions it holds that
\[
\inf_{\delta\in \FRRE}\sup_{\theta\in \Theta}\Risk(\theta,\delta)=\sup_{\pi\in \PM{\Theta}}\inf_{\delta\in \FRRE}\Risk(\pi,\delta).
\]

\item 
For statistical decision problems with bounded risk functions it holds that
\[
\inf_{\delta\in \FRRE}\sup_{\theta\in \Theta}\Risk(\theta,\delta)=\sup_{\pi\in \FM{\Theta}}\inf_{\delta\in \FRRE}\Risk(\pi,\delta),
\]
where $\FM{\Theta}$ denote the collection of finitely additive priors. 

\item 
For statistical decision problems with a totally bounded metric space as the parameter space and Lipschitz risk functions it holds that
\[
\inf_{\delta\in \FRRE}\sup_{\theta\in \Theta}\Risk(\theta,\delta)=\sup_{\pi\in \FM{\Theta}}\inf_{\delta\in \FRRE}\Risk(\pi,\delta)=\sup_{\pi\in \PM{\Theta}}\inf_{\delta\in \FRRE}\Risk(\pi,\delta).
\]
Moreover, there exists a sequence $\{\pi_n\}_{n\in \Nats}$ of finitely supported priors such that $\lim_{n\to \infty}\inf_{\delta\in \FRRE}\Risk(\pi_n,\delta)=\inf_{\delta\in \FRRE}\sup_{\theta\in \Theta}\Risk(\theta,\delta)$. 
\end{enumerate}

The third minimax theorem contains a convergence result which echoes with the fact that most existence results obtained from nonstandard analysis also produce a convergence result.  
Using these minimax theorems, we also establish connections between minimaxity and Bayesian optimality. 

\section{Preliminaries} \label{secpre}
A (non-sequential) statistical decision problem is defined in terms of
a \emph{parameter} space $\Theta$,  each element of which represents a possible state of nature;
a set $\AS$ of \defn{actions} available to the statistician;
a function $\Loss : \Theta \times \AS \to \NNReals$ characterizing the \defn{loss} associated with taking action $a \in \AS$ in state $\theta \in \Theta$;
and finally, a family $P=(P_\theta)_{\theta \in \Theta}$ of probability measures on a
\emph{sample} space $X$ with a common $\sigma$-algebra of measurable sets.
On the basis of an observation from $P_{\theta}$ for some unknown element $\theta \in \Theta$,
the statistician decides to take a (potentially randomized) action $a$, and then suffers the loss $\Loss(\theta,a)$.

To be precise, having fixed a $\sigma$-algebra on the space $\AS$ of actions,
every possible response by the statistician is
captured by a \defn{(randomized) decision procedure},
i.e., a map $\delta$ from $X$ to the space $\ProbMeasures{\AS}$ of probability measures $\AS$ such that the maps $x \mapsto \int_{\AS} \Loss(\theta,a) (\delta(x))(\dee a)$ are measurable for every $\theta\in \Theta$. 
As is customary, we will write $\delta(x,A)$ for $(\delta(x))(A)$.
The expected loss, or \defn{risk}, of the statistician associated with following a decision procedure $\delta$
conditional on the state $\theta$,
is
\[
\Risk_{\delta}(\theta)
= \Risk(\theta,\delta)
= \int_{X} \Bigl [ \int_{\AS} \Loss(\theta,a) \delta(x,\dee a) \Bigr ] \Model_{\theta}(\dee x).
\]
As the maps $x \mapsto \int_{\AS} \Loss(\theta,a) \delta(x,\dee a)$ are measurable for every $\theta \in \Theta$, the risk functions are well-defined
as functions to the extended reals. 
Let $\FRRE$ denote the set of randomized decision procedures. 
Note that it is possible that there exists a $\delta\in \FRRE$ such that $\Risk_{\delta}(\theta)=\infty$ for all $\theta\in \Theta$.

The set $\FRRE$ may be viewed as a convex subset of a vector space.
In particular, for all $\delta_1,\dotsc,\delta_n \in \FRRE$ and $p_1,\dotsc,p_n \in \NNReals$ with $\sum_i p_i = 1$,
define $\sum_i p_i \delta_i : X \to \ProbMeasures{\AS}$ by $(\sum_i p_i \delta_i)(x) = \sum_i p_i \delta_i(x)$ for $x \in X$.
Then $\Risk(\theta, \sum_i p_i \delta_i) = \sum_i p_i \, \Risk(\theta,\delta_i) < \infty$, and so we see that $\sum_i p_i \delta_i \in \FRRE$ and $r(\theta,\argdot)$ is a linear function on $\FRRE$ for every $\theta \in \Theta$.
For a subset $D \subseteq \FRRE$, let $\conv(D)$ denote the set of all finite convex combinations of decision procedures $\delta \in D$.
A subset $D\subseteq \FRRE$ is convex if $D=\conv(D)$. 

A decision procedure $\delta \in \FRRE$ is \defn{non-randomized}
if, for all $x \in X$, there exists $d(x) \in \AS$ such that $\delta(x,A) = 1$ if and only if $ d(x) \in A$, for all measurable sets $A \subseteq \AS$.
Let $\FiniteRiskEstimators \subseteq \FRRE$ denote the subset of all non-randomized decision procedures.
Under mild measurability assumptions,
every $\delta \in \FiniteRiskEstimators$ can be associated with a map $x \mapsto d(x)$ from $X$ to $\AS$
for which the risk satisfies
$
\Risk(\theta,\delta) = \int_{X} \Loss(\theta,d(x)) \Model_{\theta}(\dee x).
$

Consider now the Bayesian framework, in which one adopts a \defn{prior}, i.e., a probability measure $\pi$ defined on some $\sigma$-algebra on $\Theta$. We use $\ProbMeasures{\Theta}$ to denote the set of probability measures on $\Theta$. The Bayes risk of a randomized decision procedure $\delta\in \FRRE$ with respect to some prior $\pi\in \ProbMeasures{\Theta}$ is
\[
\Risk(\pi,\delta)=\int_{\Theta}\Risk(\theta,\delta)\pi(\dee \theta).
\]

\begin{definition}\label{lstfav}
A prior distribution $\pi_0$ is said to be least favorable if 
\[
\inf_{\delta\in \FRRE}\Risk(\pi_0,\delta)=\sup_{\pi\in \ProbMeasures{\Theta}}\inf_{\delta\in \FRRE}\Risk(\pi,\delta).
\]
\end{definition}

The concept of least favorable priors is closely related to the minimax theorem, the finite version of which is stated below.

\begin{theorem}[{\citep[][Thm.~2.9.1]{Ferguson}}]\label{minimaxthm}
Let a statistical decision problem be given where $\Theta$ is finite. Then we have
\[\label{faminimaxeq}
\inf_{\delta\in \FRRE}\sup_{\theta\in \Theta}\Risk(\theta,\delta)=\sup_{\pi\in \ProbMeasures{\Theta}}\inf_{\delta\in \FRRE}\Risk(\pi,\delta),
\]
and there exists a least favorable prior $\pi_0$. 
\end{theorem}

It is mathematically simpler to work with non-randomized decision procedures. However, \cref{minimaxthm} fails without randomized decision procedures. This is related to the fact that Nash equilibria in mixed strategies exist in far more games than Nash equilibria in pure strategies. Nonetheless, we point out the following relation between randomized decision procedures and non-randomized decision procedures. 
\[
&\inf_{\delta\in \FiniteRiskEstimators}\sup_{\theta\in \Theta}\Risk(\theta,\delta)\geq \inf_{\delta\in \FRRE}\sup_{\theta\in \Theta}\Risk(\theta,\delta)\\
&\sup_{\pi\in \PM{\Theta}}\inf_{\delta\in \FiniteRiskEstimators}\Risk(\pi,\delta)\geq \sup_{\pi\in \PM{\Theta}}\inf_{\delta\in \FRRE}\Risk(\pi,\delta).
\]

We conclude this section with the following example showing that \cref{minimaxthm} (specifically \cref{faminimaxeq}) need not hold for statistical decision problems with infinite parameter spaces. This counterexample can be found in \citep{Ferguson}. 

\begin{example}\label{canexample}
Let $X=\{0\}$ and $\Theta=\AS=\{\frac{1}{n}: n\in \Nats\}$. Let 
\[\Loss(\theta,a) = \left\{
  \begin{array}{lr}
    1, & \theta < a,\\
    0, & \theta=a,\\
    -1, & \theta>a.
  \end{array}
\right.
\]

\begin{claim}\label{neqclaim}
For every decision procedure $\delta\in \FRRE$, we have $\sup_{\pi\in \PM{\Theta}}\Risk(\pi,\delta)=1$. 
For every $\pi\in \PM{\Theta}$, we have $\inf_{\delta\in \FiniteRiskEstimators}\Risk(\pi,\delta)=-1$.
\end{claim}
\begin{proof}
Suppose $\delta\in \FRRE$ and a positive $\epsilon\in \Reals$ are given. 
Find a finite set $K\subset \AS$ such that $\delta(0,K)>1-\epsilon$;
moreover, find $\theta_{0}\in \Theta$ such that $\theta_0<k$ for every $k\in K$.  
Let $\pi_0$ be a prior that is degenerate at $\theta_0$. 
Then we have 
\[
\int \Risk(\theta,\delta) \dee \pi_0&=\Risk(\theta_0,\delta)\\
&=\int_{K}\Loss(\theta,a)\delta(0,\dee a)+\int_{\AS\setminus K}\Loss(\theta,a)\delta(0,\dee a)\\
&>1-\epsilon-\epsilon=1-2\epsilon.
\]
As $\epsilon$ was arbitrary, we have $\sup_{\pi\in \PM{\Theta}}\Risk(\pi,\delta)=1$.

Now let $\pi\in \PM{\Theta}$ and $\epsilon>0$ be given. 
Then there exists a finite set $B\subset \Theta$ such that $\pi(B)>1-\epsilon$. 
Choose $\theta_0$ such that $\theta_0<\theta$ for every $\theta\in B$. Define a nonrandomized decision procedure $\delta_0$ that puts all its mass at $d_0(0) = \theta_0$.  
Then we have 
\[
\int \Risk(\theta,\delta_0)\dee \pi&=\int_{B}\Risk(\theta,\delta_0) \dee \pi+\int_{\Theta\setminus B}\Risk(\theta,\delta_0) \dee \pi\\
&< -(1-\epsilon)+\epsilon=2\epsilon-1.
\]
As $\epsilon$ was arbitrary, we have $\inf_{\delta\in \FiniteRiskEstimators}\Risk(\pi,\delta)=-1$.
\end{proof}

Thus, we have $\inf_{\delta\in \FRRE}\sup_{\theta\in \Theta}\Risk(\theta,\delta)\neq \sup_{\pi\in {\PM{\Theta}}}\inf_{\delta\in \FiniteRiskEstimators}{\Risk(\pi,\delta)}$ which means that the minimax equality fails. 

It is usually easier for a least favorable prior to exists than for the minimax equality to hold. It is easy to see that every prior is a least favorable prior in the statistical decision problem given in \cref{canexample}. However, there are examples in statistical decision theory where least favorable priors do not exist. 
\end{example}
There exists a rich literature on minimax theorems for statistical decision problems with infinite parameter spaces. 
We give proofs of several more general (well-known) minimax theorems in \cref{pushdownsec}. 

\subsection{Notation from Nonstandard Analysis}

In this paper, we work within the standard decision theoretical framework but carry out that work using nonstandard analysis, a powerful machinery derived from mathematical logic. For those who are not familiar with nonstandard analysis, \citet[][App.~A]{DR2017} provide a review tailored to their statistical application.
\citet{NSAA97,NDV,NAW} provide thorough introductions.

We briefly introduce the setting and notation from nonstandard analysis. 
We use $\NSE{}$ to denote the nonstandard extension map taking elements, sets, functions, relations, etc., to their nonstandard counterparts.
In particular, $\HReals$ and $\NSE{\Nats}$ denote the nonstandard extensions of the reals and natural numbers, respectively.
We will always assume that $\NSE x = x$ for $x \in \Reals \cup \Nats \cup \Theta$.

An element $r\in \HReals$ is \emph{infinite} or \emph{unlimited} if $r>n$ for every $n\in \Nats$ and is \emph{finite} or \emph{limited} otherwise. 
Given a topological space $(X,\topology)$,
the monad of a point $x\in X$ is the set $\bigcap_{x\in U\in \topology}\NSE{U}$.
An element $x\in \NSE{X}$ is \emph{near-standard} if it is in the monad of some $y\in X$. 
We say $y$ is the standard part of $x$ and write $y=\ST(x)$. 
We use $\NS{\NSE{X}}$ to denote the collection of near-standard elements of $\NSE{X}$ 
and we say $\NS{\NSE{X}}$ is the \emph{near-standard part} of $\NSE{X}$. 
The standard part map $\ST$ is a function from $\NS{\NSE{X}}$ to $X$, taking near-standard elements to their standard parts.
In both cases, the notation elides the underlying space $Y$ and the topology $T$, 
both of which will always be clear from context.
For a metric space $(X,d)$, two elements $x,y\in \NSE{X}$ are \emph{infinitely close} if $\NSE{d}(x,y)\approx 0$. 
There exists a close connection between finite elements and near-standard elements. In particular, we quote the following well-known result:
\begin{theorem}\label{fanear}
An element $r\in \HReals$ is finite if and only if it is near-standard. 
\end{theorem}

We shall have opportunity to work with a variant of the usual standard part map: By the ``usual'' map we mean
\[
\ST{}\colon \NS{\HReals} \to \Reals
\]
where $\HReals$ is (of course) equipped with its usual topology.
We will also make use of the extended reals $\EReals$, that is, the usual two-point compactification of $\Reals$, with underlying set $\Reals\cup\{\infty,-\infty\}$, familiar from measure theory and real analysis. 
In this case it holds that $\NS{\NSE{\EReals}} = \NSE{\EReals}$ and $\NSE{\EReals} = \overline{\HReals}$, where on the right of the last equality, 
we mean the space obtained by applying the construction of the two-point compactification internally to $\HReals$. This space has underlying set
$\HReals\cup\{\infty,-\infty\}$ and the monad of $\infty$ (resp., $-\infty$) is the sets of positive (resp., negative) infinite hyperreals.
We shall write $\EST$ for the corresponding standard part map:
\[
\EST{}\colon  \NSE{\EReals} \to \EReals.
\]
This map can be described simply as follows: For $\tilde r \in \NSE{\EReals}$,
\[
\EST(\tilde r) = \begin{cases} 
			\ST(\tilde r)	& \text{ if $\tilde r$ is finite,}\\
			\infty			&\text{ if $\tilde r$ is unlimited and $\tilde r > 0$,}\\
			-\infty		&\text{ if $\tilde r$ is unlimited and $\tilde r < 0$.}
\end{cases}
\]

For two internal sets $A,B$, we use $\IFuncs{A}{B}$ to denote the collection of internal functions from $B$ to $A$. 
Given a standard set $X$, we use ${^{\sigma}X}$ to denote the set $\{\NSE{x}: x\in X\}$, called the standard part copy of $X$.

Let $X$ be a topological space endowed with Borel $\sigma$-algebra $\BorelSets X$ and 
let $\FM{X}$ denote the collection of all finitely additive probability measures on $(X,\BorelSets X)$. 
An internal probability measure $\mu$ on $(\NSE{X},\NSE{\BorelSets X})$ is an element of $\NSE{\FM{X}}$. 
Namely, an internal probability measure $\mu$ on $(\NSE{X},\NSE{\BorelSets X})$ 
is an internal function from $\NSE{\BorelSets X}\to \NSE{[0,1]}$ such that
\begin{enumerate}
\item $\mu(\emptyset)=0$;
\item $\mu(\NSE{X})=1$; and
\item $\mu$ is hyperfinitely additive.
\end{enumerate}
The Loeb space of the internal probability space $(\NSE{X},\NSE{\BorelSets X}, \mu)$ is a countably additive probability space $(\NSE{X},\Loeb{\NSE{\BorelSets X}}, \Loeb{\mu})$ such that 
\[
\Loeb{\NSE{\BorelSets X}}=\{A\subset \NSE{X}|(\forall \epsilon>0)(\exists A_i,A_o\in \NSE{\BorelSets X})(A_i\subset A\subset A_o\wedge \mu(A_o\setminus A_i)<\epsilon)\}.
\]
and 
\[
\Loeb{\mu}(A)=\sup\{\ST(\mu(A_i))|A_i\subset A,A_i\in \NSE{\BorelSets X}\}=\inf\{\ST(\mu(A_o))|A_o\supset A,A_o\in \NSE{\BorelSets X}\}.
\]

The remarkable utility of nonstandard analysis is embodied by the \emph{transfer principle}, which asserts that a first order sentence is true in the standard model 
if and only if it is true in the nonstandard extension after all parameters have been replaced by their image under the star map. 
Finally, given a cardinal number $\kappa$, a nonstandard model is called $\kappa$-saturated if the following condition holds:
For any family $\cF$ of cardinality less than $\kappa$ consisting of internal sets which moreover has the finite intersection property, the total intersection of $\cF$ is non-empty. In this paper, we assume our nonstandard model is as saturated as we need, that is, $\kappa$-saturated for some large enough $\kappa$.

\section{The Nonstandard Minimax Theorem}\label{secnsmini}
In this section, we present a nonstandard minimax theorem for arbitrary statistical decision problems under very mild conditions. We fix a standard statistical decision problem $(X,\Theta,\AS, \Loss,\Model)$. Let $\Theta$ be equipped with a $\sigma$-algebra $\cF$ that contains all singletons. 

Recall that for infinite $r\in \NSE{\Reals}$, we have by definition $\EST(r)=\infty$. 

\begin{theorem}\label{minimaxgtr}
For every $D\subseteq \FRRE$,
we have 
\[
\inf_{\delta\in D}\sup_{\theta\in \Theta}\Risk(\theta,\delta)\geq \sup_{\Pi\in \NSE{\ProbMeasures{\Theta}}}\inf_{\delta\in D}\EST(\NSE{\Risk}(\Pi,\NSE{\delta})).
\]
\end{theorem}
\begin{proof}
Fix a set $D\subseteq \FRRE$. 
Let $\delta_0\in D$ and $\Pi_0\in \NSE{\ProbMeasures{\Theta}}$ be arbitrary. To finish the proof, it is sufficient to show that
\[
\sup_{\theta\in \Theta}\Risk(\theta,\delta_0)\geq \inf_{\delta\in D}\EST(\NSE{\Risk}(\Pi_0,\NSE{\delta})).
\]
If $\sup_{\theta\in \Theta}\Risk(\theta,\delta_0)=\infty$, then the result is obvious. 
Therefore, we may suppose $\sup_{\theta\in \Theta}\Risk(\theta,\delta_0)\in \Reals$. 
Then
\[
\sup_{\theta\in \Theta}\Risk(\theta,\delta_0)=\sup_{\theta\in \NSE{\Theta}}\NSE{\Risk}(\theta,\NSE{\delta_0})\geq \NSE{\Risk}(\Pi_0,\NSE{\delta_0}).
\]
The last inequality above still holds if we take standard parts, and thus
\[
\sup_{\theta\in \Theta}\Risk(\theta,\delta_0) \geq \ST(\NSE{\Risk}(\Pi_0,\NSE{\delta_0}))\geq \inf_{\delta\in D}\EST(\NSE{\Risk}(\Pi_0,\NSE{\delta})). 
\]
\end{proof}

We will also need the following lemma; the proof is identical to that of \cite[Lemma~3.6]{DRS2021}.
\begin{lemma}\label{l.Theta_0}
Fix $v \in \Reals$.
Suppose $D$ is a finite subset of $\cD$ such that for each $\delta \in \conv(D)$ there is $\theta \in \Theta$ with $\Risk(\theta,\delta)>v$.
Then there is a finite set $\Theta_0 \subseteq \Theta$ such that
for each $\delta \in \conv(D)$ there is $\theta \in \Theta_0$ with $\Risk(\theta,\delta)>v$.
\end{lemma}

The following hyperplane separation theorem is well-known.  

\begin{theorem}\label{hyperplanethm}
Fix any $n \in \Nats$. If $S_{1},S_{2}$ are two disjoint convex subsets of $\Reals^n$, then there exists $p\in \Reals^n \setminus \{\vzero\}$ such that for all $v_1\in S_{1}$ and $v_{2}\in S_{2}$, we have
$\IP{p}{v_{1}} \geq \IP{p}{v_{2}}$.
\end{theorem}

An internal prior is an element of $\NSE{\ProbMeasures{\Theta}}$. 
\begin{definition}
Let $D$ be a subset of $\FRRE$. 
An internal prior $\Pi_0\in \NSE{\ProbMeasures{\Theta}}$ is said to be an $S$-least favorable prior with respect to $D$ if 
\[
\inf_{\delta\in D}\EST(\sRisk(\Pi_0,\sdelta))=\sup_{\Pi\in \NSE{\ProbMeasures{\Theta}}}\inf_{\delta\in D}\EST(\sRisk(\Pi,\sdelta))\in \Reals.
\]
An $S$-least favorable prior with respect to $\FRRE$ will simply
be referred to as an $S$-least favorable prior. 
\end{definition}

We now establish the main result of this section. 
Recall that, for every $D\subseteq \FRRE$, we use $\Ext{D}$
to denote the set $\{\NSE{\delta}: \delta\in D\}$. 
For an arbitrary set $A$, we use $\FiniteSubsets{A}$ to denote the collection of all finite subsets of $A$.

\begin{theorem}\label{nsminimax}
For any convex subset $\cC$ of $\FRRE$ the nonstandard minimax theorem holds:
\[
\inf_{\delta\in \cC}\sup_{\theta\in \Theta}\Risk(\theta,\delta)=\sup_{\Pi\in \NSE{\ProbMeasures{\Theta}}}\inf_{\delta\in \cC}\EST({\NSE{\Risk}(\Pi,\NSE{\delta})})\label{nsequality}.
\]
Moreover, there exists an $S$-least favorable prior $\Pi_0$ with respect to $\cC$, that is, the supremum on the right-hand side is realized or in other words, $\sup$ can be replaced by $\max$ on the right-hand side of the above equation.
\end{theorem}
\begin{proof}
By \cref{minimaxgtr}, we have
\[
\inf_{\delta\in \cC}\sup_{\theta\in \Theta}\Risk(\theta,\delta)\geq\sup_{\Pi\in \NSE{\ProbMeasures{\Theta}}}\inf_{\delta\in \cC}\ST({\NSE{\Risk}(\Pi,\NSE{\delta})}).
\]
Thus, it is sufficient to show that
\[
\inf_{\delta\in \cC}\sup_{\theta\in \Theta}\Risk(\theta,\delta)\leq\sup_{\Pi\in \NSE{\ProbMeasures{\Theta}}}\inf_{\delta\in \cC}\ST({\NSE{\Risk}(\Pi,\NSE{\delta})}).
\]
Let 
\[
V:=\inf_{\delta\in \cC}\sup_{\theta\in \Theta}\Risk(\theta,\delta)
\]
where we explicitly allow $V = \infty$.
For every $v  \in \Reals$ such that $v < V$ and every $D \in \FiniteSubsets{\cC}$,
define 
$\phi_{v}^{D}(\pi)$ to be the formula
\[
\pi \in \ProbMeasures{\Theta} \;\land \;
\inf_{\delta \in D} \Risk(\pi, \delta) > v.
\]
We show that the collection of formulas 
\begin{equation}\label{e.coll}
\left\{ \phi_{v}^{D}(\pi) : v \in \Reals \land v < V \land  D \in \FiniteSubsets{\cC}\right\}
\end{equation}
is finitely satisfiable. Since this collection is closed under finite conjunctions, it is clear we only need to show that each formula in this collection is satisfiable.

Let therefore $v  \in \Reals$ such that $v < V$ and $D \in \FiniteSubsets{\cC}$ be given.
Since $D$ is finite, by \cref{l.Theta_0} we can find $\Theta_0 \in \FiniteSubsets{\Theta}$ such that
\begin{equation}\label{e.conv}
v < \inf_{\delta \in \conv(D)} \sup_{\theta \in\Theta_0} \Risk(\theta,\delta).
\end{equation}
Let $n$ be the cardinality of $\Theta_0$. Recall that $\Reals^{\Theta_0}$ denotes
the set of functions from $\Theta_0$ to $\Reals$, and thus $\Reals^{\Theta_0}$
is just an isomorphic copy of euclidean space $\Reals^n$.
Define
\[
Q_{v}=\{x \in \Reals^{\Theta_0} : (\forall \theta \in \Theta_0)\; x(\theta) \leq v\}.
\]
Clearly, $Q_{v}$ is a convex set. 

For any $\delta \in \cC$, define the $v$-regularized risk function $r^v_\delta$ as follows:
\[
r^v_\delta (\theta) = 
\begin{cases}
\Risk(\theta, \delta) &\text{ if $\Risk(\theta, \delta) < \infty$;}\\
v + 1 &\text{ otherwise.}
\end{cases}
\]
Clearly, $r^v_\delta \in \Reals^{\Theta}$ for every $\delta \in \cC$.
For any set $\cC_0 \subseteq \cC$, define the $v$-regularized risk set $\RS{v,\cC_0}$ induced by $\cC_0$ to be the set of regularized risk functions of procedures in $\cC_0$, restricted to $\Theta_0$; in other words,
\[
\RS{v,\cC_0}=\{ r^v_\delta \upharpoonright \Theta_0 : \delta \in \cC_0 \}.
\]
 Clearly, $\RS{v,\cC_0} \subseteq \Reals^{\Theta_0}$.
\begin{claim}\label{finitedisjoint}
$Q_{v}\cap \RS{v,\conv(D)}=\emptyset$. 
\end{claim}
\begin{proof}
The claim obvious from \eqref{e.conv} and the definitions.
\end{proof}

As both $Q_v$ and $\RS{v,\conv(D)}$ are convex sets, by \cref{hyperplanethm} and \cref{finitedisjoint}, there is a nontrivial hyperplane which separates them; that is, there is $p \in \Reals^{\Theta_0}$ such that for all $x \in Q_v$ and $d \in \RS{v,\cC}$, 
\begin{equation}\label{e.sep}
\langle p, x\rangle \leq \langle p, d \rangle.
\end{equation}
We will now produce a prior $\pi$ from $p$ in the usual manner. First, observe the following:
\begin{claim}
$p(\theta)\geq 0$ for all $\theta\in \Theta$.
\end{claim}
\begin{proof}
Suppose that the claim is false, i.e., $p(\theta_0)<0$ for some $\theta_{0}\in \Theta$.
By finding a point $x_0$ in $\bigcup_{n\in \Nats}Q_v$ whose value at $\theta_0$ is negative and arbitrarily small,
clearly $\IP{p}{x_0}$ can be made  arbitrary large, a contradiction as this quantity is bounded above by $\IP{p}{d_0}$.
\end{proof} 
By normalizing $p \in \Reals^{\Theta_0}$ we obtain $\pi \in \ProbMeasures{\Theta_0}$;
Since 
\[
\IP{p}{r^v_\delta} \leq \Risk(\pi,\delta)
\]
and since the maximum of the left-hand side of \eqref{e.sep} is $v$, we have found $\pi$ such that
$\phi_{v}^{D}(\pi)$ holds, and hence the collection from \eqref{e.coll} is finitely satisfiable.

For every $D \in \FiniteSubsets{\Ext{\cC}}$
and $v \in \Reals$ with $v < V$,
$\NSE{\phi_{n}^{D}(\Pi)}$ is the formula
\[
\Pi \in \NSE{\ProbMeasures{\Theta}} \;\land \;
\inf_{\delta \in D} \NSE{\Risk(\Pi, \NSE\delta)} > v.
\]
By the saturation principle,
there exists  $\Pi_{0}$ satisfying all such formulas $\NSE{\phi_{n}^{D}(\Pi)}$ simultaneously.
That is, for every $\delta \in \cC$  and every $v \in \Reals$ with $v < V$,
we have  
$\NSE{\Risk(\Pi_0, \NSE\delta)} > v$, 
and hence
\[
V \leq \inf_{\delta \in \cC} \EST \big(\NSE{\Risk(\Pi_0, \NSE\delta)}\big).
\]
Thus, $\Pi_0$ is an $S$-least favorable prior with respect to $\cC$, and $\Pi_0$ realizes the standard minimax value, completing the proof. 
\end{proof}

Within the standard statistical decision theory framework, it is possible for a decision problem to have infinite minimax value while having finite maximin value, as is shown in the following example. 

\begin{example}\label{densecounter}
Let $X=\{0\}$, $\Theta=\cA=\Reals$ and let the loss function
\[\Loss(\theta,a) = \left\{
  \begin{array}{lr}
    |\theta-a|, & |\theta-a|\in \Rationals\\
    0, & \text{otherwise}.
  \end{array}
\right.
\]
Let $\Theta$ be equipped with the usual Lebesgue $\sigma$-algebra. 
Let $\cC$ be the collection of all decision procedures $\delta$ such that $\delta(0)$ is a countably supported probability measure on $\cA$. 
Note that $\cC$ is a convex subset of $\FRRE$. 
For every $\delta\in \cC$, there exists $a_0\in \cA$ such that $\delta(0,\{a_0\})>0$.
As there exists arbitrarily large $\theta\in \Theta$ such that $|\theta-a_0|\in \Rationals$,
we have
\[
\sup_{\theta\in \Theta}\Risk(\theta,\delta)&=\sup_{\theta\in \Theta}\int \Loss(\theta,a)\delta(0,\dee a)\\
&\geq \sup_{\theta\in \Theta}|\theta-a_0|\delta(0,\{a_0\})=\infty.
\]
Hence, we have $\inf_{\delta\in \cC}\sup_{\theta\in \Theta}\Risk(\theta,\delta)=\infty$. 

We now show that the maximin value is finite. 
It is sufficient to show that, for every $\pi\in \PM{\Theta}$, there exists $\delta\in \cC$ such that $\Risk(\pi,\delta)=0$. 
Let $\pi\in \PM{\Theta}$ be given. For every $a\in \Reals$, define $[a]=\{r\in \Reals: |r-a|\in \Rationals\}$. 
It is easy to see that there exists an uncountable set $A\subset \Reals$ such that $\Reals=\bigcup_{a\in A}[a]$ and $[a]\cap [b]=\emptyset$ for distinct $a,b\in A$. 
As $\pi$ is a probability measure, there exists $a_1\in A$ such that $\pi([a_1])=0$. 
Let $\delta_1\in \FRRE$ be the non-randomized decision procedure such that $\delta_{1}(0,\{a_1\})=1$. 
Note that $\delta_{1}\in \cC$ and $\Risk(\theta,\delta_1)=\Loss(\theta,a_1)=0$ for every $\theta\not\in [a_1]$.  
We have
\[
\Risk(\pi,\delta_1)&=\int_{\Reals}\Risk(\theta,\delta_1)\pi(\dee \theta)\\
&=\int_{[a_1]}\Risk(\theta,\delta_1)\pi(\dee \theta)+\int_{\Reals\setminus [a_1]}\Risk(\theta,\delta_1)\pi(\dee \theta)=0.
\]

It is also interesting to see a $\NSE{}$prior which witnesses that the nonstandard minimax theorem holds. 
To construct such a prior explicitly, let $\Delta_0, \hdots, \Delta_K$ be a hyperfinite sequence of $\NSE{}$decision procedures 
such that $\Ext{\cD} \subseteq \{\Delta_0, \hdots, \Delta_K\}$.
For each $J \in \{0, \hdots, K\}$, find $\tilde \theta_J \in \NSE{\Theta}$ 
such that 
\[
\tilde w_J := \Delta_K(0,\{\tilde \theta_J\}) > 0.
\]
Find $N \in \NSE{\Nats}$ large enough so that
\[
N \cdot \left( \min_{J} \tilde w_J \right) \cdot \frac{1}{K+1}  \text{ is unlimited}
\]
and let $\Pi$ be the $\NSE{}$probability measure which gives equal mass to every point in 
\[
\{\bar \theta_J + N \colon J \in \NSE{\Nats}, 0 \leq J \leq K\}.
\]
Then, for each $J \in \{0, \hdots, K\}$ it holds that
\[
\NSE{\Risk}(\Pi,\Delta_J) \geq \Loss(\tilde\theta_J, \tilde \theta_J + N) \cdot \Delta_J(0,\{\tilde \theta_J\}) \cdot \Pi(\{\tilde \theta_J + N\})  = N \cdot \tilde w_J \cdot \frac{1}{K+1}
\]
and thus by choice of $N$, $\EST\big( \NSE{\Risk}(\Pi,\Delta_J)\big) = \infty$.
\end{example}

\subsection{Nonstandard Bayes Optimality of Minimax Procedures}
A decision procedure $\delta_0$ is said to be \emph{minimax} with respect to a collection $\cC\subset \FRRE$ of decision procedures if $\sup_{\theta\in \Theta}\Risk(\theta,\delta_0)=\inf_{\delta\in \cC}\sup_{\theta\in \Theta}\Risk(\theta,\delta)$. A natural question to ask is: when are minimax decision procedures Bayes with respect to some prior distributions? A simple answer to this question is: If the (standard) minimax theorem holds and there is a least favorable prior $\pi$, then minimax decision procedures are Bayes procedures with respect to $\pi$. However, for an arbitrary statistical decision problem, the minimax theorem may fail and least favorable prior need not necessarily exist. In this section, using \cref{nsminimax}, we show that every minimax decision procedure is nonstandard Bayes with respect to some nonstandard prior. 

The following definition of Bayes optimality is taken from \citet{DR2017}. 
\begin{definition}
Let $\cC\subset \FRRE$ and let $\Pi_0$ be an internal prior. 
A decision procedure $\delta_0$ is nonstandard Bayes under $\Pi_0$ with respect $\cC$ if $\sRisk(\Pi_0,\NSE{\delta_0})$ is hyperfinite and, 
for all $\delta\in \cC$, we have $\sRisk(\Pi_0,\NSE{\delta_0})\lessapprox \sRisk(\Pi_0,\NSE{\delta})$.
\end{definition}

The following theorem shows that every minimax decision procedure is nonstandard Bayes under the least favorable prior, provided that the minimax value is finite. 

\begin{theorem}\label{mininsbayes}
Let $\cC$ be a convex subset of $\FRRE$. 
Suppose $\inf_{\delta\in \FRRE}\sup_{\theta\in \Theta}\Risk(\theta,\delta)<\infty$.
Let $\delta_0$ be a minimax decision procedure with respect to $\cC$ and let $\Pi_0$ be an $S$-least favorable prior with respect to $\cC$.
Then $\delta_0$ is nonstandard Bayes under $\Pi_0$ with respect to $\cC$. 
\end{theorem}
\begin{proof} 
By assumption, we have $\inf_{\delta\in \cC}\sup_{\theta\in \Theta}\Risk(\theta,\delta)=\sup_{\theta\in \Theta}\Risk(\theta,\delta_0) < \infty$ and 
\[
\sup_{\Pi\in \NSE{\PM{\Theta}}}\inf_{\delta\in \cC}\EST(\sRisk(\Pi,\NSE{\delta}))=\inf_{\delta\in \cC}\EST(\sRisk(\Pi_0,\sdelta)). 
\]
From these equalities and \cref{nsminimax}, we infer $\sup_{\theta\in \Theta}\Risk(\theta,\delta_0)=\inf_{\delta\in \cC}\EST(\sRisk(\Pi_0,\sdelta))$. 
By the transfer principle, $\sup_{\theta\in \Theta}\Risk(\theta,\delta_0)=\sup_{\tilde \theta\in \NSE{\Theta}}\sRisk(\tilde \theta,\NSE{\delta_0})\geq \ST(\sRisk(\Pi_0,\NSE{\delta_0}))$. 
Thus, we have 
\[
\ST(\sRisk(\Pi_0,\NSE{\delta_0})) = \inf_{\delta\in \cC}\EST(\sRisk(\Pi_0,\sdelta)), 
\]
showing that $\delta_0$ is nonstandard Bayes under $\Pi_0$ with respect to $\cC$. 
\end{proof} 

We find it interesting to note the following equivalence:
\begin{lemma}\label{hyriskfinite}
Let $D$ be a subset of $\FRRE$.
Then $\inf_{\delta\in D}\sup_{\theta\in \Theta}\Risk(\theta,\delta)<\infty$ if and only if there is $\delta_1\in D$ such that $\NSE{\Risk}(\tilde\theta,\NSE{\delta_1})\in \NS{\NSE{\Reals}}$ for all $\tilde\theta\in \NSE{\Theta}$. 
\end{lemma}
\begin{proof}
First, suppose $\inf_{\delta\in D}\sup_{\theta\in \Theta}\Risk(\theta,\delta)=v\in \Reals$. 
Let $\epsilon\in \PosReals$ be given. 
Find $\delta_1\in D$ such that $\sup_{\theta\in \Theta}\Risk(\theta,\delta_1)<v+\epsilon$. 
Thus, we have the following sentence:
\[
(\forall \theta\in \Theta)(\Risk(\theta,\delta_1)<v+\epsilon).
\]
By the transfer principle, we have
\[
(\forall \theta\in \NSE{\Theta})(\NSE{\Risk}(\theta,\NSE{\delta_1})<v+\epsilon).
\]
By \cref{fanear}, we have $\NSE{\Risk}(\theta,\NSE{\delta_1})\in \NS{\NSE{\Reals}}$ for all $\theta\in \NSE{\Theta}$. 

For the other direction, let us suppose that $\delta_1\in D$ is such that $\NSE{\Risk}(\theta,\NSE{\delta_1})\in \NS{\NSE{\Reals}}$ for all $\theta\in \NSE{\Theta}$. 
Then there must be $M \in \Reals$ such that for all $\theta \in \Theta$, $\Risk(\theta,\delta_1) \leq M$; for otherwise, 
by overspill---or equivalently, by $\sigma$-saturatedness---there must be $\tilde \theta \in \NSE{\Theta}$ such that
$\NSE{\Risk}(\tilde \theta, \NSE{\delta_1})$ is infinite, which stands in contradiction to our assumption.
Thus, the risk function of $\delta_1$ is totally bounded by $M$, and $\inf_{\delta\in D}\sup_{\theta\in \Theta}\Risk(\theta,\delta) \leq M$.
\end{proof}

\section{Push-down Results for Standard Statistical Decision Problems}\label{pushdownsec}
As we can see from \cref{secnsmini}, the nonstandard minimax theorem (\cref{nsminimax}) establishes the equivalence between the minimax and maximin values, provided that one is willing to adopt nonstandard priors. 
In this section, we will derive several standard minimax theorems as corollaries to \cref{nsminimax}. 

Throughout the present section, the parameter space $\Theta$ is equipped with the $\sigma$-algebra $\BorelSets{\Theta}$ of Borel subsets of $\Theta$.
The idea is to ``push down" an $S$-least favorable prior $\Pi\in \NSE{\ProbMeasures{\TTheta}}$ to obtain a standard (finitely-additive) prior $\pi$ on $(\Theta,\BorelSets {\Theta})$.  
Under moderate regularity conditions, we establish strong connections between the internal risk of $\NSE{\delta}$ with respect to $\Pi$ and the risk of $\delta$ with respect to $\pi$, for every decision procedure $\delta\in \FRRE$.

Although \cref{nsminimax} applies to every convex sub-collection of decision procedures, we focus on $\FRRE$, the set of all randomized decision procedures in this section. Every result in this section in fact holds for arbitrary convex subsets of $\FRRE$. 

\subsection{A Minimax Theorem on Compact Parameter Space}
A natural generalization of the minimax theorem with finite parameter space (\cref{minimaxthm}) is to the case where the parameter space $\Theta$ is compact. Every internal probability measure $\Pi$ on $(\NSE{\Theta},\NSE{\BorelSets {\Theta}})$ is associated with a standard countably additive prior $\pi$ in this case. 
In fact, using \cref{nsminimax}, we establish a standard minimax theorem for statistical decision problems with compact parameter spaces and upper semi-continuous risk functions. 

We start by quoting the following technical result from nonstandard analysis. 

\begin{lemma}[{\citep[][Thm.~13.4.1]{NDV}}]\label{pushdown}
Let $\gY$ be a compact Hausdorff space equipped with Borel $\sigma$-algebra $\BorelSets {\gY}$,
let $\nu$ be an internal probability measure defined on $(\NSE{\gY}, \NSE{\BorelSets{\gY}})$,
and let $\cC=\{C\subset \gY: \ST^{-1}(C) \in \overline{\NSE{\BorelSets{\gY}}}\}$.
Define a probability measure $\pd{\nu}$ on the sets $\cC$ by $\pd{\nu}(C)=\Loeb{\nu}(\ST^{-1}(C))$. Then $(\gY,\cC,\pd{\nu})$ is the completion of a regular Borel probability space.
\end{lemma} 

\begin{definition}\label{defpushdown}
The measure $\pd{\nu} : \cC \to [0,1]$ in \cref{pushdown} is called the \defn{external push-down} of the internal probability measure $\nu$.
\end{definition}

The following theorem establishes the strong connection between an internal probability measure and its external push-down. 

\begin{theorem}[{\citep[][Lemma.~2]{nsweak}}]\label{pdintcts}
Let $Y$ be a compact Hausdorff space equipped with the Borel $\sigma$-algebra $\BorelSets {Y}$,
let $\nu$ be an internal probability measure on $(\NSE{Y},\NSE{\BorelSets {Y}})$,
let $\pd{\nu}$ be the external push-down of $\nu$, and let $f: Y\to \Reals$ be a bounded measurable function.
Define $g: \NSE{Y}\to \Reals$ by $g(s)=f(\ST(s))$. Then we have $\int f \dee \pd{\nu}=\int g \dee \Loeb{\nu}$. 
\end{theorem}

We are interested in statistical decision problems with upper semi-continuous risk functions. 
Recall the following nonstandard characterization of upper semi-continuous functions. 

\begin{lemma}\label{nsupper}
Let $(Y,d)$ be a metric space. 
A function $f: Y\to \Reals$ is an upper semi-continuous function if and only if $\NSE{f}(y)\lessapprox f(\ST(y))$ for every $y\in \NS{\NSE{Y}}$. 
\end{lemma}
\begin{proof}
Suppose that $f$ is upper semi-continuous. 
Pick an element $y\in \NS{\NSE{Y}}$ and a positive $\epsilon\in \Reals$. 
Let $y_0=\ST(y)$. 
There exists a positive $r\in \Reals$ such that
\[
f(y_0)\geq \sup_{a\in Y}\{f(a): d(a,y_0)<r\}-\epsilon
\]
By the transfer principle, we have $\NSE{f}(y_0)\geq \sup_{a\in \NSE{Y}}\{\NSE{f}(a): \NSE{d}(a,y_0)<r\}-\epsilon$. 
As $y_0\approx y$, we have $\NSE{f}(y_0)\geq \NSE{f}(y)-\epsilon$.
As our choice of $\epsilon$ is arbitrary, we have $f(y_0)\gtrapprox \NSE{f}(y)$. 

Suppose that $f$ is not upper semi-continuous.
Then there exists $y_0\in Y$ such that $f$ is not upper semi-continuous at $y_0$. 
Thus, there exists a positive $\epsilon\in \Reals$ such that the following first-order statement is true:
\[
(\forall r>0)(\exists a\in Y)(d(a,y_0)<r\wedge f(y_0)<f(a)-\epsilon)
\]
By the transfer principle, there exists a $y\in \NSE{Y}$ such that $y\approx y_0$ and $f(y_0)<\NSE{f}(y)-\epsilon$, completing the proof. 
\end{proof}

We also need the following definition from nonstandard integration theory. 

\begin{definition}\label{sintegral}
Let $(\Omega,\Gamma,P)$ be an internal probability space, 
let $F: \Omega\to \NSE{\Reals}$ be an internally integrable function such that $\ST \circ F$ is defined $\Loeb{P}$-almost surely. 
The function $F$ is $S$-integrable with respect to $P$ if $\ST(\int |F| \dee P)$ exists and
\[
\ST(\int |F| \dee P)=\lim_{n\to \infty} \ST(\int |F_n| \dee P),
\]
where for $n\in \Nats$, $F_n=\min\{F,n\}$ when $F\geq 0$ and $F_n=\max\{F,-n\}$ when $F\leq 0$.
\end{definition}

The following theorem is an important consequence of $S$-integrability.

\begin{theorem}[{\citep[][Thm.~4.6.2]{NSAA97}}]\label{intheory}
Let $(\Omega,\Gamma,P)$ be an internal probability space, 
let $F: \Omega\to \NSE{\Reals}$ be an internally integrable function such that $\ST(F)$ exists $\Loeb{P}$-almost surely. 
The function $F$ is $S$-integrable with respect to $P$ if and only if $\ST{F}$ is $\Loeb{P}$-integrable and $\int F \dee P\approx \int \ST(F) \dee \Loeb{P}$. 
\end{theorem}

With these prerequisites at our disposal, we are able to prove the following key lemma. 

\begin{lemma}\label{pdintupper}
Let $Y$ be a compact Hausdorff space equipped with Borel $\sigma$-algebra $\BorelSets {Y}$,
let $\nu$ be an internal probability measure on $(\NSE{Y},\NSE{\BorelSets {Y}})$,
let $\pd{\nu}$ be the external push-down of $\nu$, 
and let $f: Y\to \Reals$ be an upper semi-continuous function such that $\NSE{f}$ is $S$-integrable with respect to $\nu$. 
Then we have $\int |f| \dee \pd{\nu}\gtrapprox \int |\NSE{f}| \dee \nu$.
\end{lemma}
\begin{proof}
Pick a positive $\epsilon\in \Reals$. 
As $\NSE{f}$ is $S$-integrable with respect to $\nu$, there exists $n\in \Nats$ such that 
$\int |\NSE{f}| \dee \nu-\int |(\NSE{f})_n| \dee \nu<\epsilon$. 
Note that $(\NSE{f})_n=\NSE{(f_n)}$. 
Define $g_n: \NSE{Y}\to \Reals$ by $g(s)=f_{n}(\ST(s))$.
By \cref{pdintcts}, we have $\int |f_{n}| \dee \pd{\nu}=\int |g|\dee \Loeb{\nu}$. 
By \cref{nsupper}, we have $g(s)\gtrapprox (\NSE{f})_{n}(s)$ for all $s\in \NSE{Y}$. 
By \cref{intheory}, we have 
\[
\int |g|\dee \Loeb{\nu}\geq \int |\ST(\NSE{f}_{n})| \dee \Loeb{\nu}\approx \int |(\NSE{f})_n| \dee \nu.
\]
Thus, we have $\int |f| \dee \pd{\nu}\geq \int |f_{n}| \dee \pd{\nu}\gtrapprox \int |\NSE{f}| \dee \nu-\epsilon$. 
As $\epsilon$ is arbitrary, we have the desired result.  
\end{proof}

In order to establish a standard minimax theorem for statistical decision problems with compact parameter space, we impose the following condition on risk functions. 

\begin{condition}{RC}[risk upper semi-continuity]
\label{assumptionrc}
For every $\delta\in \FRRE$, the risk function $\Risk(\argdot,\delta)$ is upper semi-continuous on $\Theta$.
\end{condition} 

We now prove the following standard minimax theorem on statistical decision problems with compact parameter space and upper semi-continuous risk functions. 
Our theorem is a slight generalization of a general minimax theorem by \citet{sion}.  

\begin{theorem}[{\citep[][Thm.~4.2]{sion}}]\label{extpdminimax}
Suppose $\inf_{\delta\in \FRRE}\sup_{\theta\in \Theta}\Risk(\theta,\delta)<\infty$, the parameter space $\Theta$ is compact and \cref{assumptionrc} holds. Then we have 
\[
\inf_{\delta\in \FRRE}\sup_{\theta\in \Theta}\Risk(\theta,\delta)=\sup_{\pi\in {\PM{\Theta}}}\inf_{\delta\in \FRRE}{\Risk(\pi,\delta)},
\]
and there exists a least favorable prior $\pi_0$.
\end{theorem}
\begin{proof} 
Note that we always have $\sup_{\pi\in {\PM{\Theta}}}\inf_{\delta\in \FRRE}{\Risk(\pi,\delta)}\leq \inf_{\delta\in \FRRE}\sup_{\theta\in \Theta}\Risk(\theta,\delta)$. 
As $\inf_{\delta\in \FRRE}\sup_{\theta\in \Theta}\Risk(\theta,\delta)<\infty$, by \cref{nsminimax}, there exists an $S$-least favorable prior $\mu\in \NSE{\PM{\TTheta}}$ such that
\[
\inf_{\delta\in \FRRE}\ST(\sRisk(\mu,\sdelta))=
\sup_{\Pi\in \NSE{\PM{\TTheta}}}\inf_{\delta\in \FRRE}\ST(\sRisk(\Pi,\sdelta))
=\inf_{\delta\in \FRRE}\sup_{\theta\in \Theta}\Risk(\theta,\delta).
\]
For every $\Pi\in \NSE{\PM{\TTheta}}$, by \cref{pushdown} and the fact that $X$ is compact, $\pd{\Pi}$ is an element of $\PM{\Theta}$. 
By \cref{pdintupper}, we have $\sRisk(\mu,\sdelta)\lessapprox \Risk(\pd{\mu},\delta)$ for  every $\delta\in \FRRE$. 
Thus, we have $\inf_{\delta\in \FRRE}\ST(\sRisk(\mu,\sdelta))\leq \inf_{\delta\in \FRRE}\Risk(\pd{\mu},\delta)\leq \inf_{\delta\in \FRRE}\sup_{\theta\in \Theta}\Risk(\theta,\delta)$.
Hence we have
\[
\inf_{\delta\in \FRRE}\Risk(\pd{\mu},\delta)=\inf_{\delta\in \FRRE}\sup_{\theta\in \Theta}\Risk(\theta,\delta)=\sup_{\pi\in {\PM{\Theta}}}\inf_{\delta\in \FRRE}{\Risk(\pi,\delta)}.
\]
Thus, $\pd{\mu}$ is a least favorable prior, completing the proof. 
\end{proof}

As the standard minimax equality holds and a least favorable prior exists, we immediately obtain the following result:
\begin{theorem}
Suppose $\inf_{\delta\in \FRRE}\sup_{\theta\in \Theta}\Risk(\theta,\delta)<\infty$, the parameter space $\Theta$ is compact and \cref{assumptionrc} holds. Then there exists a least favorable prior $\pi_0$. Moreover, if $\delta_0$ is a minimax decision procedure, then $\delta_0$ is Bayes with respect to $\pi_0$. 
\end{theorem}

Recall \cref{canexample} from \cref{secpre}. As the parameter space $\Theta=\{\frac{1}{n}: n\in \Nats\}$ is discrete, \cref{assumptionrc} holds automatically. Hence, \cref{extpdminimax} fails for statistical decision problems with continuous risk functions but non-compact parameter space.

\subsection{Finitely additive Minimax Theorem}
As we can see from \cref{canexample}, for statistical decision problems with non-compact parameter space, 
the minimax equality can fail unless one enlarges the set of prior probability distributions. As the literature stands, these natural relaxations include generalized priors, sequences of priors and finitely additive priors. In this section, assuming risk functions are bounded, we give a nonstandard proof of a finitely additive minimax theorem, originally due to \citet{oddsmaking}. 

For non-compact Borel measurable spaces, the external push-down of an internal probability measure may fail to be a probability measure. For example, the external push-down of the internal probability measure concentrating on some infinite number, will assign zero mass to every subset of $\Reals$. 
To overcome this problem, we now describe an alternative way to push down internal probability measures.  

\begin{definition} \label{intpushdown}
Let $X$ be a set and $\nu$ be an internal probability measure on $(\NSE{X},\NSE{\PowerSet(X)})$ where $\PowerSet(X)$ denotes the power set of $X$. The \emph{internal push-down} of $\nu$ is the function $\ipd{\nu}: \PowerSet(X)\mapsto [0,1]$ defined by $\ipd{\nu}(A)=\ST({\nu(\NSE{A})})$.
\end{definition}

The following lemma follows immediately from \cref{intpushdown}.

\begin{lemma} \label{pushdownlemma}
Let $\nu$ be an internal probability measure on $(\NSE{X},\NSE{\PowerSet(X)})$. Then its internal push-down $\ipd{\nu}$ is a finitely additive probability measure on $(X,\PowerSet(X))$.
\end{lemma}

We abuse the notation to use $\FM{\Theta}$ to denote the collection of finitely additive probability measures on $(X,\PowerSet(X))$. 
We quote the following theorem which establishes a close connection between an internal probability measure and its internal push down. 
For completeness, we include the proof in the paper. 

\begin{theorem}[{\citep[][Thm.~5.4]{faprob}}]\label{intpushdowneq}
Let $\nu$ be an internal probability measure on $(\NSE{X},\NSE{\PowerSet(X)})$ and let $f: X\mapsto \Reals$ be a bounded function. Then we have
\[
\int_{\NSE{X}} \NSE{f}(x) \nu(\dee x)\approx \int_{X} f(x) \ipd{\nu}(\dee x).
\]
\end{theorem}
\begin{proof}
Fix $\epsilon\in \PosReals$. Let $\{K_1,K_2,\dotsc,K_n\}$ be a partition of a large enough interval of $\Reals$ containing the range of $f$ such that every $K_i\in \{K_1,K_2,\dotsc,K_n\}$ is an interval with diameters no greater than $\epsilon$. For $i\leq n$, let $F_i=f^{-1}(K_i)$. Then $\{F_1,\dotsc,F_n\}\subset \BorelSets X$ is a partition of $X$ such that $|f(x)-f(x')|<\epsilon$ for every $x,x'\in F_i$ for every $i=1,2,\dotsc,n$.
Pick $x_{i}\in F_i$ for every $i=1,2,\dotsc,n$.
Define $g: X\mapsto \Reals$ by letting $g(x)=f(x_i)$ if $x\in F_i$ for every $i=1,2,\dotsc,n$.
Then $g$ is a simple bounded real valued function on $X$.
Thus, both $g$ and $f$ are integrable with respect to $\ipd{\nu}$.

We now have
$|\ \int_{\NSE{X}} \NSE{f}(x) \nu(\dee x)-\int_{X} f(x) \ipd{\nu}(\dee x)\ |\leq |\int \NSE{f}(x) \nu(\dee x)- \int \NSE{g}(x) \nu(\dee x)|
+|\int \NSE{g}(x) \nu(\dee x)-\int g(x) \ipd{\nu}(\dee x)|+|\int g(x) \ipd{\nu}(\dee x)-\int f(x) \ipd{\nu}(\dee x)|$
where all internal integrals are over $\NSE{X}$ and all standard integrals are over $X$.

By the transfer principle, we have $|\NSE{f}(x)-\NSE{f}(x')|<\epsilon$ for every $x,x'\in \NSE{F_i}$ for every $i=1,2,\dotsc,n$. Thus, we have $|\NSE{f}(x)-\NSE{g}(x)|<\epsilon$ for all $x\in \NSE{X}$.
Thus, we have $|\int_{\NSE{X}} \NSE{f}(x) \nu(\dee x)- \int_{\NSE{X}} \NSE{g}(x) \nu(\dee x)|\leq \int_{\NSE{X}}|\NSE{f}(x)-\NSE{g}(x)|\nu(\dee x)<\epsilon$. Similarly, we have $|\int_{X} g(x) \ipd{\nu}(\dee x)-\int_{X} f(x) \ipd{\nu}(\dee x)|<\epsilon$.  We now consider the term $|\int_{\NSE{X}} \NSE{g}(x) \nu(\dee x)-\int_{X} g(x) \ipd{\nu}(\dee x)|$:
\[
\int_{\NSE{X}} \NSE{g}(x) \nu(\dee x)&=\sum_{i=1}^{n}\int_{\NSE{F_i}} \NSE{g}(x) \nu(\dee x)\\
&=\sum_{i=1}^{n}\NSE{f}(x_i)\nu(\NSE{F_i})
=\sum_{i=1}^{n}f(x_i)\nu(\NSE{F_i})
\approx \sum_{i=1}^{n}f(x_i)\ipd{\nu}(F_i)\\
&=\sum_{i=1}^{n}\int_{F_i} g(x) \ipd{\nu}(\dee x)=\int_{X} g(x) \ipd{\nu}(\dee x)
\]
Thus, we have $|\int_{\NSE{X}} \NSE{f}(x) \nu(\dee x)-\int_{X} f(x) \ipd{\nu}(\dee x)|\lessapprox 2\epsilon$. As $\epsilon$ is arbitrary, we have $\int_{\NSE{X}} \NSE{f}(x) \nu(\dee x)\approx \int_{X} f(x) \ipd{\nu}(\dee x)$.
\end{proof}

We impose the following condition on the risk functions.

\begin{condition}{RB}[risk boundedness]\label{assumptionrb}
For every $\delta\in \FRRE$,
$\Risk(\argdot,\delta)$ is a bounded function on $\Theta$.
\end{condition}

The following result is a direct consequence of \cref{intpushdowneq}.

\begin{theorem}\label{pushdowneq}
Suppose \cref{assumptionrb} holds. 
Let $\Pi$ be an internal prior on $T_{\Theta}$ and let $\ipd{\Pi}$ be the internal push-down of $\Pi$.
Let $\delta_0 \in \FRRE$ be a standard decision procedure.
Then $\Risk(\argdot,\delta_{0})$ is $\ipd{\Pi}$-integrable and $\Risk(\ipd{\Pi},\delta_0)\approx \sRisk(\Pi,\sdelta_0)$,
i.e., the Bayes risk under $\ipd{\Pi}$ of $\delta_{0}$ is within an infinitesimal of the nonstandard Bayes risk under $\Pi$ of $\sdelta_{0}$.
\end{theorem}

We extend the meaning of being \emph{least favorable} to finitely additive priors in 
the natural manner.
\begin{definition}
A finitely additive prior $\pi_0$ is said to be least favorable if
\[
\inf_{\delta\in \FRRE}\Risk(\pi_0,\delta)=\sup_{\pi\in {\FM{\Theta}}}\inf_{\delta\in \FRRE}{\Risk(\pi,\delta)}.
\]
\end{definition}
 
We now give a proof to the following result which is originally due to Sudderth and Heath. It is an direct consequence of \cref{nsminimax,pushdowneq}.

\begin{theorem}[{\citep[][Thm.~2]{oddsmaking}}]\label{intpdminimax}
Suppose \cref{assumptionrb} holds. Then we have 
\[
\inf_{\delta\in \FRRE}\sup_{\theta\in \Theta}\Risk(\theta,\delta)=\sup_{\pi\in {\FM{\Theta}}}\inf_{\delta\in \FRRE}{\Risk(\pi,\delta)},
\]
and there exists a least favorable finitely additive prior $\pi_0$.
\end{theorem}
\begin{proof}
By \cref{assumptionrb}, we know that $\inf_{\delta\in \FRRE}\sup_{\theta\in \Theta}\Risk(\theta,\delta)<\infty$. By \cref{nsminimax}, there exists an internal least favorable prior $\Pi_0$ such that
\[
\inf_{\delta\in \FRRE}\EST(\sRisk(\Pi_0,\sdelta))=\sup_{\Pi\in \NSE{\ProbMeasures{\TTheta}}}\inf_{\delta\in \FRRE}\EST(\sRisk(\Pi,\sdelta))=\inf_{\delta\in \FRRE}\sup_{\theta\in \Theta}\Risk(\theta,\delta).
\]
Let $\pi_0$ be the internal push-down of $\Pi_0$ (that is, $\pi_0=\ipd{(\Pi_0)}$). By \cref{pushdowneq}, we know that $\Risk(\pi_0,\delta)\approx \sRisk(\Pi_0,\sdelta)$ for all $\delta\in \FRRE$. Thus, we have $\inf_{\delta\in \FRRE}{\Risk(\pi_0,\delta)}=\inf_{\delta\in \FRRE}\sup_{\theta\in \Theta}\Risk(\theta,\delta)$. 

Note that we always have 
\[
\inf_{\delta\in \FRRE}{\Risk(\pi_0,\delta)}\leq \sup_{\pi\in \FM{\Theta}}\inf_{\delta\in \FRRE}\Risk(\pi,\delta))\leq \inf_{\delta\in \FRRE}\sup_{\theta\in \Theta}\Risk(\theta,\delta).
\]
Thus, we know that $\sup_{\pi\in \FM{\Theta}}\inf_{\delta\in \FRRE}\Risk(\pi,\delta))=\inf_{\delta\in \FRRE}\sup_{\theta\in \Theta}\Risk(\theta,\delta)$ and $\pi_0$ is a least favorable finitely additive prior. 
\end{proof}

The following result establishes the connection between minimax principle and finitely additive Bayes optimality. It is a direct consequence of \cref{intpdminimax}.
\begin{theorem}
Suppose \cref{assumptionrb} holds. 
Then there exists a least favorable finitely additive prior $\pi_0$. 
Moreover, if $\delta_0$ is a minimax decision procedure, then $\delta_0$ is Bayes with respect to $\pi_0$. 
\end{theorem}

\begin{example}\label{canexample2}
Let us consider the statistical decision problem defined in \cref{canexample}.
We now show that there exists a least favorable finitely additive prior $\pi_0$ such that
\[
\inf_{\delta\in \FRRE}\sup_{\theta\in \Theta}\Risk(\theta,\delta)=\sup_{\pi\in {\FM{\Theta}}}\inf_{\delta\in \FRRE}{\Risk(\pi,\delta)}=\inf_{\delta\in \FRRE}\Risk(\pi_0,\delta). 
\]
Recall from \cref{canexample}, we know that $\inf_{\delta\in \FRRE}\sup_{\theta\in \Theta}\Risk(\theta,\delta)=1$. 
Thus, it is sufficient to show that $\sup_{\pi\in {\FM{\Theta}}}\inf_{\delta\in \FRRE}{\Risk(\pi,\delta)}=\inf_{\delta\in \FRRE}\Risk(\pi_0,\delta)=1$.

Let $\Pi$ be an internal probability measure concentrating on $\frac{1}{K}$ for some infinite $K\in \NSE{\Nats}$ and let $\pi_0=\ipd{\Pi}$.
By \cref{pushdownlemma}, $\pi_0$ defines a finitely additive probability measure on $(\Theta,\BorelSets {\Theta})$. 
Pick an arbitrary $\delta\in \FRRE$ and a positive $\epsilon\in \Reals$. 
There exists a finite set $B\subset \cA$ such that such that $\delta(0,B)>1-\epsilon$. 
Then we have
\[
\sRisk(\Pi,\NSE{\delta})&=\int \sRisk(\theta,\NSE{\delta})\dee \Pi\\
&=\NSE{\Risk}\left(\frac{1}{K},\NSE{\delta}\right)=\int_{\NSE{\cA}}\NSE{\Loss}\left(\frac{1}{K},a\right)\NSE{\delta}(0,\dee a)\\
&=\int_{B}\NSE{\Loss}\left(\frac{1}{K},a\right)\NSE{\delta}(0,\dee a)+\int_{\NSE{\cA}\setminus B}\NSE{\Loss}\left(\frac{1}{K},a\right)\NSE{\delta}(0,\dee a)\\
&>1-\epsilon-\epsilon=1-2\epsilon.
\]
As the choice of $\epsilon$ is arbitrary, we have $\sRisk(\Pi,\NSE{\delta})=1$. 
By \cref{pushdowneq}, we have $\Risk(\pi_0,\delta)=\sRisk(\Pi,\NSE{\delta})=1$. Hence we have 
\[
\inf_{\delta\in \FRRE}\sup_{\theta\in \Theta}\Risk(\theta,\delta)=\inf_{\delta\in \FRRE}\Risk(\pi_0,\delta)=1, 
\]
completing the proof. 
\end{example}

\subsection{Least Favorable Sequence of Priors}

An alternative way to deal with the failure of \cref{minimaxthm} is to approximate the maximin value with a sequence of priors. Sequences of priors can be seen as approximations of generalized priors and finitely additive priors. For example, the Lebesgue measure on $\Reals$ can be approximated by a sequence of uniform distributions on $[-n,n]$, with $n \to \infty$. In the latter case, as pointed out in \citet{faprob}, under moderate conditions, every finitely additive prior can be viewed as a weak limit of a sequence of countably additive priors. On the other hand, by taking the external push-down of the nonstandard extension of a finitely additive probability measure, one obtains a countably additive probability measure on the compactification. By using this strong connection between finitely additive priors and sequences of countable additive priors, we obtain a standard minimax theorem for statistical decision problems with non-compact parameter spaces. 

\begin{definition}[least favorable sequence of priors]\label{deflfseq}
A sequence of prior distributions $\{\pi_n\}$ is least favorable if for every prior distribution $\pi$ we have 
\[
\Risk_{\pi}\leq \Risk=\lim_{n\to \infty}\Risk_{\pi_n}
\]
where $\Risk_{\pi}=\inf\{\Risk(\pi,\delta): \delta\in \FRRE\}$. 
\end{definition}

In order to formally study the relation between finitely additive priors and sequences of countably additive priors, we introduce the Wasserstein distance on the space of finitely additive probability measures. 

\begin{definition}\label{defwass}
Let $X$ be a metric space endowed with Borel $\sigma$-algebra $\BorelSets X$ and let $\mu,\nu$ be two finitely additive probability measures on $(X,\BorelSets X)$. Let $k\in \Nats$. Then the $k$-Wasserstein distance between $\mu$ and $\nu$ is 
\[
W_{k}(\mu,\nu)=\sup\{|\int f \dee \mu-\int f\dee \nu|: f\in \Lip{k}{X}\}
\]
where $\Lip{k}{X}$ denotes the collection of all $k$-Lipschitz functions from $X$ to $\Reals$. 
\end{definition}

We quote the following result in nonstandard analysis. 

\begin{lemma}[{\citep[][Thm.~4.16]{faprob}}]\label{pdcloseipd}
Suppose $X$ is a bounded $\sigma$-compact metric space and $\nu$ is an internal probability measure on $(\NSE{X},\NSE{\BorelSets X})$.
Let $\pd{\nu}$ and $\ipd{\nu}$ denote the external push-down and internal push-down of $\nu$, respectively.
Suppose $\pd{\nu}$ is a countably additive probability measure on $(X,\BorelSets X)$.
Then $W_{k}(\pd{\nu},\ipd{\nu})=0$ for all $k\in \Nats$. 
\end{lemma}

Thus, even if the Wasserstein distance defines a metric on the set of probability measures, it is merely a pseudo-metric on the set of finitely additive probability measures. 
A (not necessarily compact) metric space $X$ can often be embedded into a compact metric space $\comp{X}$, called the compactification of $X$. However, the metric on $X$ may be different when viewed as a subset of $\comp{X}$. The following theorem gives a sufficient and necessary condition for the existence of a compactification $\comp{X}$ into which $X$ embeds isometrically. 

\begin{lemma}\label{isocomp}
A metric space $X$ is isometrically embeddable into a compact metric space $\comp{X}$ if and only if it is totally bounded. 
\end{lemma}
\begin{proof}
Suppose $X$ is totally bounded. 
Then the completion of $X$ is a complete and totally bounded metric space hence is compact. 
Conversely, suppose $X$ is isometrically embedded into the compact metric space $\comp{X}$.
As $\comp{X}$ is totally bounded, so is $X$. 
\end{proof} 

Moreover, as $\comp{X}$ can be taken to be the completion of $X$, every real-valued bounded uniformly continuous function on $X$ admits an unique uniformly continuous extension to $\comp{X}$. Bounded uniformly continuous functions play an essential role in weak convergence, as we shall see in the following theorem. 

\begin{theorem}[{\citep[][Thm.~5.8]{faprob}}]\label{wkconverge}
Suppose $X$ is a separable metric space endowed with Borel $\sigma$-algebra $\BorelSets X$. 
Let $P$ be a finitely additive probability measure on $(X,\BorelSets X)$. 
There exists a sequence $\{P_n\}_{n\in \Nats}$ of finitely supported probability measures on $(X,\BorelSets X)$ such that
\[
\lim_{n\to \infty}\int f \dee P_n=\int f \dee P
\]
for every real-valued bounded uniformly continuous function $f$ if and only if $X$ is totally bounded. 
\end{theorem}

\cref{wkconverge} is sharp in the sense that it fails for merely bounded continuous functions. 
We are now at the place to present the following minimax theorem. 

\begin{theorem}\label{seqminimax}
Suppose $\inf_{\delta\in \FRRE}\sup_{\theta\in \Theta}\Risk(\theta,\delta)<\infty$. 
Suppose, in addition, the parameter space $\Theta$ is a totally bounded separable metric space, and there exists $k\in \Nats$ such that $\Risk(\cdot,\delta)\in \Lip{k}{\Theta}$ for every $\delta\in \FRRE$.  
Then we have
\[
\inf_{\delta\in \FRRE}\sup_{\theta\in \Theta}\Risk(\theta,\delta)=\sup_{\pi\in \PM{\Theta}}\inf_{\delta\in \FRRE}\Risk(\pi,\delta)=\sup_{\pi\in \FM{\Theta}}\inf_{\delta\in \FRRE}\Risk(\pi,\delta).
\]
Moreover, there exists a sequence $\{\pi_n\}_{n\in \Nats}$ of finitely supported priors such that 
\[
\lim_{n\to \infty}\Risk_{\pi_n}=\inf_{\delta\in \FRRE}\sup_{\theta\in \Theta}\Risk(\theta,\delta).
\]
\end{theorem}
\begin{proof}
As $\Theta$ is totally bounded and $\Risk(\cdot,\delta)\in \Lip{k}{\Theta}$ for every $\delta\in \FRRE$,
\cref{assumptionrb} holds.
By \cref{intpdminimax}, there exists a least favorable finitely additive prior $\pi_0$ such that
\[
\inf_{\delta\in \FRRE}\sup_{\theta\in \Theta}\Risk(\theta,\delta)=\sup_{\pi\in {\FM{\Theta}}}\inf_{\delta\in \FRRE}{\Risk(\pi,\delta)}=\inf_{\delta\in \FRRE}\Risk(\pi_0,\delta).
\]
By \cref{wkconverge}, there exists a sequence $\{\pi_n\}_{n\in \Nats}$ of finitely supported priors such that
$\lim_{n\to \infty} \int f \dee \pi_n=\int f \dee \pi_0$ for every real-valued bounded uniformly continuous function $f$. 
Let $\comp{\Theta}$ be the completion of $\Theta$. By \cref{isocomp}, $\comp{\Theta}$ is a compactification of $\Theta$. 
Note that $\pi_0$ as well as $\{\pi_n\}_{n\in \Nats}$ can be naturally extend to (finitely additive) probability measures on $(\comp{\Theta},\BorelSets {\comp{\Theta}})$. Let $P=\pd{(\NSE{\pi_0})}$. By \cref{pdcloseipd}, we see that $P$ is a countably additive probability measure on $(\comp{\Theta},\BorelSets {\comp{\Theta}})$ with $W_{k}(P,\pi_0)=0$. 

\begin{claim}
The sequence $\{\pi_n\}_{n\in \Nats}$ converges to $P$ weakly. 
\end{claim}
\begin{proof}
Pick a bounded continuous function $f: \comp{\Theta}\to \Reals$. 
Then the restriction of $f$ to $\Theta$ is a real-valued bounded uniformly continuous function. 
Thus, we have
\[
\lim_{n\to \infty}\int_{\comp{\Theta}} f \dee \pi_n&=\lim_{n\to \infty} (\int_{\Theta} f \dee \pi_n+\int_{\comp{\Theta}\setminus\Theta} f \dee \pi_n)\\
&=\int_{\Theta} f \dee \pi_0+\int_{\comp{\Theta}\setminus\Theta} f \dee \pi_0=\int_{\comp{\Theta}}f \dee \pi_0.
\]
By \cref{intpushdowneq,pdintcts}, we know that $\int_{\comp{\Theta}}f \dee \pi_0=\int_{\comp{\Theta}}f \dee P$. 
Thus, we have $\lim_{n\to \infty}\int_{\comp{\Theta}} f \dee \pi_n=\int_{\comp{\Theta}}f \dee P$ hence $\pi_n$ converges to $P$ weakly. 
\end{proof}

As $\comp{\Theta}$ is compact, we have $\lim_{n\to \infty} W_k(\pi_n,P)=0$. As $\comp{X}$ is the compact metric completion of $X$, 
every $k$-Lipschitz function $f: X\to \Reals$ can be extended to a real-valued $k$-Lipschitz function on $\comp{X}$. 
Thus, we have $\lim_{n\to \infty} W_k(\pi_n,\pi_0)=0$. 
Pick a positive $\epsilon\in \Reals$.
There exists a $m\in \Nats$ such that $W_k(\pi_i,\pi_0)<\epsilon$ for every $i>m$.
In particular, we have $|\Risk(\pi_i,\delta)-\Risk(\pi_0,\delta)|<\epsilon$ for all $\delta\in \FRRE$ and all $i>m$.
Hence, we have $|\Risk_{\pi_i}-\Risk(\pi_0)|<\epsilon$ for all $i>m$, which implies that
$\lim_{n\to \infty}\Risk_{\pi_n}=\Risk_{\pi_0}$.

In conclusion, we have
\[
\inf_{\delta\in \FRRE}\sup_{\theta\in \Theta}\Risk(\theta,\delta)=
\sup_{\pi\in \PM{\Theta}}\inf_{\delta\in \FRRE}\Risk(\pi,\delta)=\sup_{\pi\in \FM{\Theta}}\inf_{\delta\in \FRRE}\Risk(\pi,\delta)
=\lim_{n\to \infty}\Risk_{\pi_n},
\]
completing the proof. 
\end{proof}

Under the assumptions of \cref{seqminimax}, there exists a sequence of least favorable priors $\{\pi_n\}_{n\in\Nats}$ such that each of them is finitely supported. In \cref{seqminimax}, the assumptions on the parameter space as well as risk functions seem quite strong. 
In the following example, we show that these assumptions are necessary.  

\begin{example}\label{canexample3}
Let us first consider the statistical decision problem defined in \cref{canexample}. 
The parameter space $\Theta$ is a totally bounded, locally compact separable metric space.
The loss function $\Loss$ is bounded continuous but not $k$-Lipschitz continuous for any $k\in \Nats$. 
By \cref{canexample2}, there exists a finitely additive prior $\pi_0$ such that $\inf_{\delta\in \FRRE}\Risk(\pi_0,\delta)=\inf_{\delta\in \FRRE}\sup_{\theta\in \Theta}\Risk(\theta,\delta)=1$. 
However, from \cref{canexample}, we have $\inf_{\delta\in \FRRE}\Risk(\pi,\delta)=-1$ for every $\pi\in \PM{\Theta}$, hence the minimax equality fails.

Let $X=\{0\}$, $\Theta=\AS=\{n: n\in \Nats\}$ and $\Model_{\theta}(\{0\})=1$ for all $\theta\in \Theta$. 
Define
\[\Loss(\theta,a) = 
  \begin{cases}
    1			&\text{ if $\theta > a+1$,}\\
    \theta-a		&\text{ if $\theta\in [a-1,a+1]$,}\\
    -1			&\text{ if $\theta<a-1$.}
  \end{cases}
\]   
The parameter space $\Theta$ is a locally compact separable metric space but not totally bounded.  
For every $\delta\in \FRRE$, the risk function $\Risk(\cdot,\delta)$ is a $1$-Lipschitz function. 
Using a similar argument as in the proof of \cref{neqclaim}, 
we can show that $\sup_{\theta\in \Theta}\Risk(\theta,\delta)=1$ for every $\delta\in \FRRE$ and $\inf_{\delta\in \FRRE}\Risk(\pi,\delta)=-1$ for all $\pi\in \PM{\Theta}$. 
Hence we have $\inf_{\delta\in \FRRE}\sup_{\theta\in \Theta}\Risk(\theta,\delta)=1$ and $\sup_{\pi\in {\PM{\Theta}}}\inf_{\delta\in \FRRE}{\Risk(\pi,\delta)}=-1$. Hence the minimax equality fails. 

Finally, we show that the finitely additive minimax equality holds and there exists a finitely additive least favorable prior. 
Let $\Pi$ be an internal prior that concentrates on $K$ for some infinite $K\in \NSE{\Nats}$. 
Using a similar argument as in \cref{canexample2},
we have $\sRisk(\Pi,\NSE{\delta})=\sRisk(K,\NSE{\delta})=1$ for every $\delta\in \FRRE$.
Let $\pi_0=\ipd{\Pi}$. By \cref{pushdowneq}, we have $\Risk(\pi_0,\delta)=\sRisk(\Pi,\NSE{\delta})=1$. 
Hence we have 
\[
\inf_{\delta\in \FRRE}\sup_{\theta\in \Theta}\Risk(\theta,\delta)=\sup_{\Pi\in \NSE{\ProbMeasures{\TTheta}}}\inf_{\delta\in \FRRE}\ST({\NSE{\Risk}(\Pi,\NSE{\delta})})=\sup_{\pi\in {\FM{\Theta}}}\inf_{\delta\in \FRRE}{\Risk(\pi,\delta)}=1.
\]

\end{example}

\printbibliography

\end{document}